\documentclass[twocolumn, a4paper]{article}
\usepackage{amsmath, amsthm, amsfonts, amssymb}

\usepackage[utf8]{inputenc}

\usepackage{mathtools, cuted}
\usepackage{microtype}
\usepackage{xcolor}

\definecolor{mylinks}{rgb}{0.0, 0.24, 0.44}
\definecolor{mycites}{rgb}{0.0, 0.44, 0.44}
\usepackage[
  colorlinks=true,
  citecolor={mycites},
  linkcolor={mylinks}
]{hyperref}

\usepackage[capitalize]{cleveref}

\usepackage{abstract}
\usepackage{authblk}
\usepackage[
  backend=biber,
  isbn=false,
  citestyle=numeric,
  url=false,
  sorting=none
]{biblatex}
\newcommand{\bs}[1]{#1}
\newcommand{\bb}[1]{}

\newcommand*{\dd}{\ensuremath{\dot{d}}}
\newcommand*{\dr}{\ensuremath{\dot{r}}}
\newcommand*{\dx}{\ensuremath{\dot{x}}}

\newtheorem{theorem}{Theorem}[section]
\newtheorem{corollary}[theorem]{Corollary}
\newtheorem{lemma}[theorem]{Lemma}
\newtheorem{assumption}[theorem]{Assumption}
\newtheorem*{remark}{Remark}

\title{Circadian Clock Model with Sequestration Repression Motif: Existence of  Periodic Orbits and Entrainment Properties}
\author[1]{
  B\"obel, Benjamin
}
\author[1]{
  Chaves, Madalena
 }
\author[1]{
    Gouzé, Jean-Luc
}
\affil[1]{Inria Centre at Université Côte d’Azur, INRAE, CNRS, Macbes team, \protect\\
Sophia Antipolis, France}
\date{}

\begin{document}
\twocolumn[
\begin{@twocolumnfalse}
\maketitle

\begin{abstract}
\noindent Protein sequestration motifs appear in many biological regulatory networks and introduce special properties into the network dynamics.
Sequestration can be described as a mode of inactivation of a given protein by its binding to a second protein to form a new complex. In this complexed form, the original protein is prevented from performing its specific functions and is thus rendered inactive.
We study a mathematical model of the mammalian circadian clock with a protein sequestration motif, which generates one of the negative feedback loops that guarantees periodic behavior.
First, the motif permits a time-scale separation, which can be used to simplify the model. We show that the simplified model admits a periodic orbit over a fairly large region of parameters.
Second, we are able to show that the sequestration motif induces a phase response curve with a very specific form, implying that external perturbations can affect the system only in a narrow window of the periodic orbit. Finally, we compare our model with a classic Goodwin oscillator, to illustrate the dynamical and robustness properties induced by the protein sequestration motif.\vspace{1cm}

\end{abstract}
\end{@twocolumnfalse}
]
\section{Introduction}
Virtually all living systems display internal rhythms that are coordinated with the geophysical day-night schedule \cite{patke_molecular_2020}. These circadian rhythms are generated on the cellular level by a complex and circular interplay of synthesis, degradation and sequestration of transcription factors, a group of regulatory proteins \cite{takahashi_transcriptional_2017}. Cellular circadian oscillators, or clocks, have a freerunning period of about 24h and possess the ability to vary this period with respect to external input.

Each type of mammalian cell exhibits the same core genetic oscillator, while its external inputs and also downstream outputs vary with its tissue type \cite{schibler_clock-talk_2015}. Common inputs for circadian clocks include light information from the eye in the central clock of the brain, body temperature, insulin levels or exercise for circadian clocks in peripheral tissues.

Circadian systems in mammals are characterized by a hierarchical structure where the suprachiasmatic nucleus (SCN) in the hypothalamus plays the central role and cells of peripheral organs are downstream \cite{schibler_clock-talk_2015}. In the healthy state, each organ exhibits internal coherence between the rhythms of its cells. This has been most prominently shown for the cells in the SCN. In vitro slices of SCN tissue show sustained coherent oscillations which do not dampen or dephase with time \cite{hastings_generation_2018}. In vitro samples of other tissues do not share this behaviour as their oscillations dampen or dephase. In mammals, time-of-day information from the SCN is passed down to oscillators in peripheral tissues through a variety of pathways such as body temperature or endocrine factors \cite{schibler_clock-talk_2015}.

In the last 30 years a variety of models have been proposed to describe the transcriptional-translational feedback loops that constitute the core of the circadian system in every cell  \cite{goldbeter_model_1995, brown_dual-feedback_2020, hesse_mathematical_2021, gonze_spontaneous_2005, almeida_transcription-based_2020}. Among them are approaches to include all known processes in order to make quantitative predictions for the medical community, as well as reduced models which take into account only the minimum amount of constituents needed to produce self-sustained oscillations.

The transcriptional-translational feedback loops in the circadian clock are based on transcription factors, i.e., proteins in the nucleus which recognize specific DNA sequences, bind to these, and influence the transcription rate of the corresponding gene \cite[402]{alberts_molecular_2022}. The DNA sequences which are bound are regulatory elements, or \textit{boxes}, that lie in the promoter region of a gene. There are three boxes which are bound by transcription factors of the core circadian system, the clock-controlled elements (CCEs): E-box, D-box and R-box (RORE). In turn, all of the genes coding for the core circadian proteins have at least one of these CCEs in their promoter region \cite{ueda_system-level_2005}.

For example, the complex CLOCK:BMAL1 is a transcriptional activator that binds to E-boxes and contributes to the expression of the gene \textit{Dbp} (among others), increasing the production of the protein DBP. DBP, in turn, is an activator of the D-box, increasing the expression of proteins REV-ERB and PER. Finally, REV-ERB is a transcriptional repressor binding to the R-box, repressing the expression of CLOCK and BMAL1.

CLOCK:BMAL1 is the core protein complex of the circadian system. Its concentration peaks at dawn and it is the major activator of clock-controlled genes, both downstream genes which are rhythmically expressed and other genes of the clock system \cite{takahashi_transcriptional_2017, rey_genome-wide_2011}.
In mice it binds over 2000 sites and, in particular, upgregulates genes which are important for metabolism \cite{rey_genome-wide_2011}.

The mammalian circadian system is made up of two interlocked negative feedback loops that affect the complex CLOCK:BMAL1.
The first negative feedback loop via REV represses transcription of gene \textit{Bmal1} and eventually contributes to a decrease in the CLOCK:BMAL1 complex. This effect is often represented by a decreasing Hill-type function of REV on the synthesis of CLOCK:BMAL1.

The second feedback is one of the major processes in the circadian clock involving the tight binding of the complex PER:CRY to the transcription factor CLOCK:BMAL1
We follow \cite{buchler_protein_2009, francois_core_2005, kim_mechanism_2012, briat_antithetic_2016} in calling this process protein sequestration by complexation, or simply \textit{sequestration}. The sequestration process inactivates the transcriptional activity of CLOCK:BMAL1.

Most models have focused on the transcriptional mechanisms of the circadian clock, in which repression can be modeled with a Michaelis-Menten or Hill type function. Sequestration of CLOCK:BMAL1 is a post-translational process which can be more suitably described by mass action kinetics. The process of protein sequestration as a repression mechanism in biological systems is well-known \cite{buchler_protein_2009, trevino-quintanilla_anti-sigma_2013}. It is known to generate ultrasensitive responses in genetic networks, comparable to strongly cooperative processes (high Hill-coefficients) \cite{buchler_molecular_2008, buchler_protein_2009}. In \cite{briat_antithetic_2016} the sequestration motif is used as the key process in the design of a biological feedback controller.

Sequestration type repression as a central process of the circadian system has been studied in the Kim-Forger model \cite{kim_molecular_2014, kim_mechanism_2012}. A derivation of their model can be found in \cite[46-49]{forger_biological_2017}. In their model the activator in the sequestration has a constant concentration which simplifies the mass-action expression into a static repression function. This static repression function has a big impact on the behaviour of a network of coupled oscillators, notably, the coupled period stays close to the freerunning period \cite{kim_protein_2016}.

In \cite{francois_model_2005}, the author models the \textit{Neurospora crassa} circadian clock, where a dynamic sequestration repression is the only negative feedback loop and finds that this repression mechanism suffices to produce oscillations in ODE models. Stable oscillations in an ODE model with only sequestration repression are also found in \cite{briat_antithetic_2016}. However, in the latter example the authors show that the mean of stochastic simulations of the same model does not produce oscillations.

In this paper, we study sequestration as a dynamical process, where both the dynamics of the activator concentration and the repressor concentration are modeled dynamically. We use a model that takes into account both major feedback loops and combine Hill type for transcriptional repression and mass action kinetics for sequestration repression. This model has been previously developed in our team and used to describe the interaction of circadian clock and cell cycle \cite{almeida_control_2020} and to study the effect of time-restricted feeding on clusters of coupled cells \cite{burckard_cycle_2022}.

We now mathematically show the existence of oscillations using the symmetry of the sequestration process. The sequestration process is symmetrical in the sense, that it has the same effect on the concentrations of the two involved chemical species. In our case, one PER:CRY complex inactivates exactly one CLOCK:BMAL1. Mathematically, sequestration is described by the same mass-action term in the dynamics of two variables and introduces therefore a symmetrical term in the model equations.

Furthermore, we employ the theory of phase responses to explain some entrainment properties of the oscillations. The Phase Response Curve (PRC) is a useful measure in theoretical models as well as real biological systems \cite{forger_biological_2017, kulkarni_sensitivity_2014, winfree_geometry_2001}. It tabulates resulting phase shifts ($\Delta \phi$) from a localized external perturbation, a pulse, as a function of perturbation timing ($\phi$) and is thus an input-output measure of the periodic steady state of a dynamical system \cite{sacre_sensitivity_2014}. An external perturbation changes amplitude and phase of an oscillator. For stable periodic orbits, amplitude changes are transient, however, while phase changes persist \cite{granada_chapter_2009}. The PRC $\Delta \phi (\phi)$ therefore captures the persistent response of an oscillator, relying on the assumption that trajectories rapidly fall back to the periodic orbit.

\cref{sec:model} introduces a reduced piecewise affine model for circadian oscillations exploiting the sequestration process. \cref{sec:existence} contains a proof for the existence of a periodic behaviour by a cyclic transition of subsystems. \cref{sec:prc} shows the phase response-curve of the model. Finally,  \cref{sec:entrainment} studies entrainment properties that result from the phase responses.

\section{A Reduced Circadian Clock Model}\label{sec:model}
\begin{figure}[hb]
    \centering
    \includegraphics[width=0.8\columnwidth]{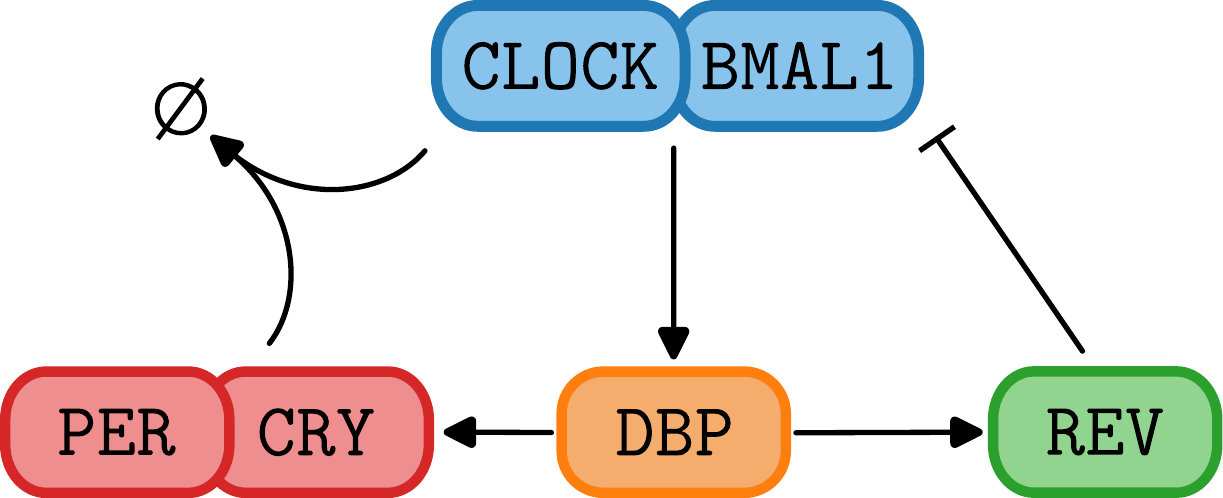}
    \caption{Scheme of the circadian clock model in \eqref{eq:sofia}. Complexes CLOCK:BMAL1 and PER:CRY and proteins DBP and REV-ERB interact forming two negative feedback loops.}
    \label{Fig1}
\end{figure}
We study a four variable ODE model that describes the transcriptional regulation and major post-translational effects of circadian rhythms in mammalian peripheral cells. It has been proposed in \cite{almeida_control_2020} as a skeleton model which describes only the main transcriptional-translational feedback loops (TTFL).
Our team built a larger model of the core circadian clock based on the occupation of the CCEs E-box, D-box and R-box and their role in the expression of the genes \textit{Bmal1}, \textit{Ror}, \textit{Rev}, \textit{Dbp}, \textit{E4bp4}, \textit{Per} and \textit{Cry} \cite{almeida_transcription-based_2020}.
The larger 8-variable model uses Michaelis-Menten type activation of the three CCEs which are incorporated in the ODE description of protein production. The complex CLOCK:BMAL1 was chosen to be represented by the rate-limiting BMAL1 dynamics, since CLOCK and BMAL1 share the same CCE (R-box). The complex PER:CRY was modeled through the complex formation of PER and CRY. By quasi-steady state assumptions on several protein concentrations, a 4-variable skeleton model was obtained. Because of the steady-state calculations, the CCE activation of PER percolates to the complex PER:CRY. As a consequence, the production term for the protein REV-ERB and the complex PER:CRY become identical in the skeleton model.
We choose to work with the skeleton model, as it emphasizes the main dynamical elements of the system and is amenable to detailed analytical study, while retaining the biological properties. See \cref{Fig1} for the scheme of our model.

The model variables are concentrations of the protein complexes CLOCK:BMAL1 and PER:CRY and proteins DBP and REV.  We group proteins of homologous genes into one variable such that variable CRY takes into account both CRY1 and CRY2. The model equation read as follows:
\begin{align}
\begin{split}\label{eq:sofia}
\text{[CLOCK:BMAL1]} \quad \dot{B} &= V_R\, h(R) - \gamma_B \, B P \\
\text{[DBP]} \quad \dot{D} &= V_B\, B - \gamma_D\, D \\
\text{[REV-ERB]} \quad \dot{R} &= V_D\, D - \gamma_R\, R \\
\text{[PER:CRY]} \quad \dot{P} &= V_D\, D - \gamma_B\, B P\\
\text{with}\quad  h(R) &= \frac{k_R^2}{k_R^2 + R^2}
\end{split}
\end{align}
where every variable describes a concentration of a protein or protein complex and has a synthesis and a degradation process. $B$ stands for [BMAL1:CLOCK], $D$ for [DBP], $R$ for [REV:ERB] and $P$ for [PER:CRY]. The degradation parameters are named $\gamma$ and the synthesis parameters $V$.

There are two negative feedback loops which center around CLOCK:BMAL1 and imply a delayed repression of the complex after each increase in concentration. The first negative feedback loop via REV acts on CLOCK:BMAL1 using a Hill-type repression. The second feedback loop includes a protein sequestration type repression where the complex PER:CRY binds to CLOCK:BMAL1 and renders it inactive. By law of mass action, the rate of the sequestration process is proportional to both concentrations.
Note, that the synthesis terms stem from a reduction of the larger 8-variable model, as described above.

Equations for $B$, $D$, $R$ in \eqref{eq:sofia} resemble a Goodwin oscillator, since they invoke the nonlinear degradation of $B$ and then a downstream cascade of linear terms. The difference in our system is the degradation part of $B$, which is not a natural degradation but the mutual sequestration of $B$ and $P$ described above.

The standard parameter values used in this article are direct descendants of the fitted 8-variable model used in \cite{almeida_transcription-based_2020}. All parameters of their large model were fitted to experimental timeseries of fluorescence intensity of a REV-ERB reporter. Afterwards the 8-variable model was reduced to \eqref{eq:sofia} as described above.

\subsection{Parameter Reduction} \label{sec:paramred}
To analyse the dynamical behaviour of model \eqref{eq:sofia} and, in particular, to investigate the contribution of the sequestration phenomenon, we first reduce the dimensionality of the parameter space. To this end, we introduce the change of variables: $B = V_R b$, $D = V_B V_R d$, $R = k_R r$, $P = V_R p$. This brings about the parameter-reduced system:
\begin{align}\label{eq:reduced_smooth}
\begin{split}
\dot{b} &= h_d(r) - \alpha \, b p\\
\dot{d} &= b - \beta\, d\\
\dot{r} &= \gamma\, d - \delta\, r\\
\dot{p} &= \epsilon\, d - \alpha \,  b p \\
\text{with}\quad &h_d(r) = \frac{1}{1+r^2}
\end{split}
\end{align}
where the new parameters are given by: $\alpha = V_R \gamma_B$ , $\beta = \gamma_D$ , $\gamma = V_B V_D V_R / k_R$ , $\delta = \gamma_R$ , $\epsilon = V_B V_D$ .

Secondly, the sequestration process introduces a symmetrical degradation term in $b$ and $p$. This leads to phase-opposition in the dynamics of variables $b$ and $p$, where $b \approx 0$ during the peak of $p$ and vice versa. The following change of variables
\begin{align}\label{eq:cov}
    x = b-p \qquad y = b+p
\end{align}
with its inverse
\begin{align}\label{eq:bp}
    b = \frac{1}{2} (x+y) \qquad p = \frac{1}{2} (y-x)
\end{align}

makes use of this symmetrical process and separates slow and fast elements. The dynamics of the transformed system read as follows:
\begin{subequations} \label{eq:separated}
\begin{align}
\dot{x} &= h_d(r) - \epsilon\, d\\
\dot{d} &= \frac{1}{2}(x+y) - \beta\, d\\
\dot{r} &= \gamma\, d - \delta\, r\\
\dot{y} &= h_d(r) + \epsilon\, d + \frac{\alpha}{2} (x^2  - y^2) \label{eq:ydot} \ .
\end{align}
\end{subequations}

Variables $x$ and $y$ no longer represent individual concentrations. As a difference of positive quantities, $x$ can take on positive and negative values.

By change of variables, the two mass-action expressions describing the sequestration are integrated into the new variable $y$. This is now the only variable containing the dynamical nonlinearity of the sequestration process.

Simulation of this system shows oscillations which are remarkably insensitive in period and amplitude to increasing the sequestration parameter $\alpha$.
This stems from the fact, that the quadratic term in $\dot{y}$ dominates the other terms when the condition $y^2=x^2$ is violated.

Note, that without synthesis processes, the difference between the two concentrations involved in the sequestration (called $x$) remains constant, as previously used in \cite{briat_antithetic_2016}. In the article, the authors also remark, that this reaction is on a faster timescale than the rest, given a suitably high sequestration parameter.
This separation of timescales is central to the following model reduction.

\subsection{The large sequestration rate limit} \label{sec:seqrate}
\begin{figure}
    \centering
    \includegraphics[width=\columnwidth]{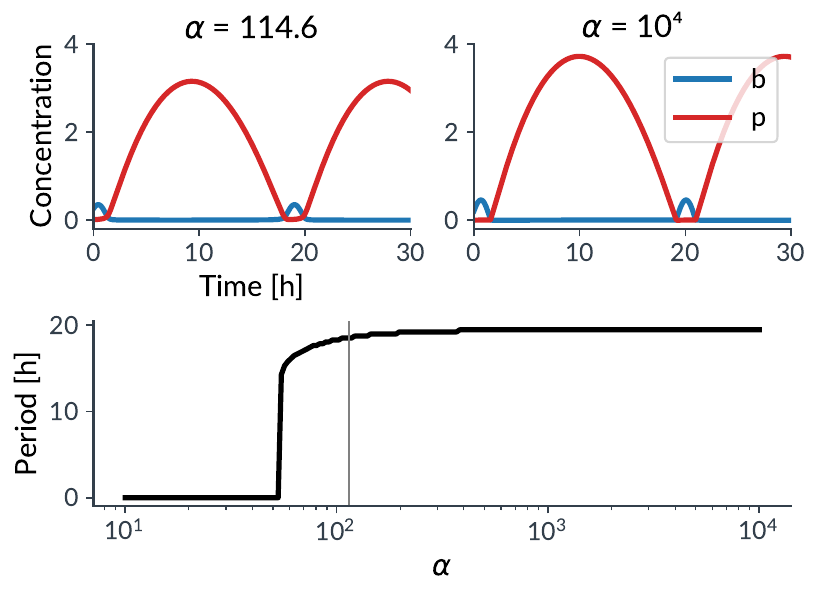}
    \caption{Influence of sequestration parameter $\alpha$ on system dynamics and period. \textit{Top:} Solutions of variables $b$ and $p$ for two different values of parameter $\alpha$, after transients. \textit{Bottom:} Period of periodic solution of model \eqref{eq:reduced_smooth} as a function of parameter $\alpha$. The standard value of $\alpha=114.6$, determined by fitting to experimental data, is shown in grey. Other parameters: $\beta = 0.156$, $\gamma = 0.15$, $\delta = 0.241$, $\epsilon = 2.698$.}
    \label{Fig2}
\end{figure}
As described in \cref{sec:paramred}, the sequestration term dominates other terms in \eqref{eq:ydot} when the timescales are appropriately separated. This is a consequence of the expression involving the difference of squares and the fact, that $x$ and $y$ take values on the same order of magnitude as the other variables.
For the small time-window when $x \approx 0 \Leftrightarrow b \approx p$, the time-scale separation does not hold for $\alpha$ below some threshold, $\alpha_{\text{thres}}$.
As long as the sequestration parameter $\alpha$ is sufficiently large, timescales remain separated and \eqref{eq:ydot} will be controlled by the sequestration term. Conversely, once timescales are separated, increasing $\alpha$ above $\alpha_{\text{thres}}$ no longer changes the dynamics. To visualize this, we show two timeseries and the period of the limit cycle solution depending on $\alpha$ (see \cref{Fig2}).

For convenience, in the following we will consider the limit $\alpha \to \infty$ as a proxy for sufficiently separated timescales. However, the separation is inherent in the model dynamics and does not stem from the choice of a limit sequestration rate.

In the limit $\alpha \to \infty$ the sequestration process is infinitely fast, so $h(r)/\alpha \ll 1$ and $(\epsilon d)/\alpha \ll 1$. We can use a quasi steady-state argument to pose that the dynamics of $y$ are dictated by: $|y| = |x|$ (see \cite{segel_quasi-steady-state_1989}). Since $y$ is the sum of positive quantities, we can write: $y=|x|$.

This limit effectively means that there can never be a coexistence of both species $b$ and $p$, since they sequester with an infinite rate. One of the concentrations has to be zero and any production of the zero-concentration species is immediately subtracted from the other one. Therefore, in the large sequestration rate limit, the composed variable $x = b-p$ describes the interlocked dynamics of $b$ and $p$ completely, with the two cases ($p=0, b>0  \Rightarrow x>0$ and $b=0, p>0 \Rightarrow x<0$).

Because of \eqref{eq:bp} and $y=|x|$, it is easy to see that:
\begin{align} \label{eq:b(x)}
    b = \frac{1}{2} (x+y) = \frac{1}{2} (x + |x|) = \begin{cases}x,\ x>0\\ 0,\ \mathrm{else} \end{cases} \eqqcolon f_b(x) \ .
\end{align}
Function $f_b(x)$ is used as the synthesis term for variable $d$ from now on.
\subsection{Piecewise affine model}
In order to further simplify the system, we replace the nonlinearly decreasing Hill-function $h_d(r)$ with a switch function $h_s(r)$ that is constant 1 below its threshold and constant 0 above it. The switch function $h_s(r)$ is the limit \mbox{$n\rightarrow \infty$} of decreasing Hill-functions. With these two approximations, our model reduces to the following piecewise affine (PWA) system with linear time-invariant subsystems:
\begin{align}\label{eq:pwa}
\begin{split}
\dot{x} &= h_s(r) - \epsilon d \\
\dot{d} &= f_b(x) - \beta d\\
\dot{r} &= \gamma d - \delta r \ \\
\text{with}\quad &h_s(r) = \begin{cases} 1, \ r< 1\\ 0, \ \text{else} \end{cases},
\end{split}
\end{align}
where $f_b(x)$ as defined in \eqref{eq:b(x)}.

The advantages of this model are that $4$ regions can be clearly identified and an explicit solution for the dynamics in each region can be written. From these solutions, we can infer possible transitions between the regions and eventually detect global solutions such as periodic orbits.

A sample trajectory of the PWA model is shown in \cref{Fig3}. Note how the dynamics of $x$ change abruptly at its trough due to the switch function $h_s(r)$.
\begin{figure}
    \centering
    \includegraphics[width=\columnwidth]{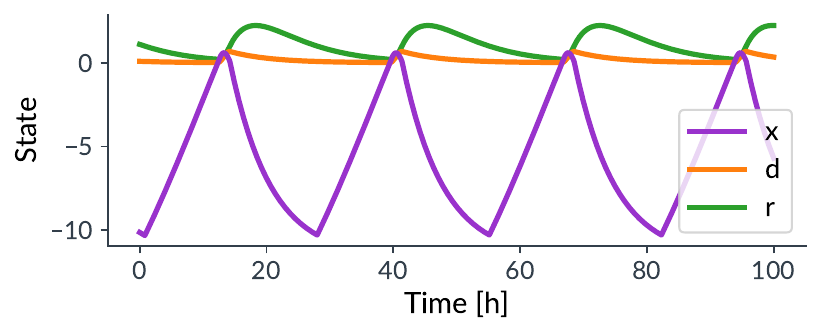}
    \caption{Timeseries of the PWA model. Parameters: $\beta = 0.156, \gamma = 0.15, \delta = 0.241, \epsilon = 2.698$.}
    \label{Fig3}
\end{figure}

\section{Existence of a Periodic Orbit}\label{sec:existence}
\begin{figure}[ht!]
    \centering
    \includegraphics[width=0.75\columnwidth]{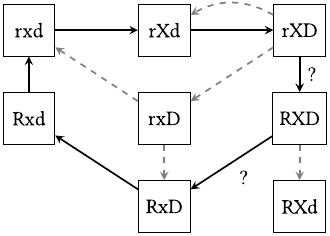}
    \caption{Scheme of the transitions for the piecewise-linear model. Possible transitions are shown as arrows between subsystems. Unwanted transitions are dashed. Transitions with a question mark will be proved in order to follow sequence \eqref{eq:sequence}. For clarity, not all possible transitions are shown.}
    \label{Fig4}
\end{figure}

\noindent It is in general hard to explicitly show the existence of a periodic orbit in systems of nonlinear ODEs of dimension larger than 2 \cite[43]{guckenheimer_nonlinear_1983}. To study the oscillatory behavior of circadian systems with two feedback loops, we resort to the PWA system \eqref{eq:pwa}. The PWA system captures both feedback loops while still being analytically tractable. In the following we show that, under certain conditions, a specific cyclic sequence of transitions between subsystems is the only possible solution of the system dynamics. The existence of a periodic orbit follows via the Brouwer fixed point theorem.

First, we define three thresholds $r^*, x^*, d^*$ to naturally divide the state space into 8 regions. In each region, the dynamics is given by one of four affine subsystems. The regions follow a notation based on whether the variables are above or below their respective thresholds, e.g., the region rXd is the subsystem where $r<r^*, x>x^*$ and $d<d^*$.

The threshold in $x$ is $x^* = 0$, the non-differentiable point in $f_b(x)$ (see \cref{eq:b(x)}). The threshold in $r$ is $r^*=1$, which is the discontinuity of the switch function $h_s(r)$. Whether $r$ can transition its threshold is defined by $\dot{r}(r=r^*) = \gamma d - \delta$. The value of variable $d$ determines whether $\dot{r}$ on the plane $r=r^*$ is bigger or smaller than 0. Therefore the value of $d$ completely determines whether the variable $r$ can transition. To define a threshold in variable $d$, we introduce $d^* = \delta / \gamma$ such that $d < d^* \Rightarrow \dot{r}(r=1) <0$ and likewise $d > d^* \Rightarrow \dot{r}(r=1) >0$.

In the following, we give sufficient conditions for a specific sequence of subsystems. Since this sequence is cyclic, solutions must exhibit oscillations.
\begin{assumption} \label{ass:dstar}
    Assume $\gamma > \epsilon \delta$, or equivalently \mbox{$1 > \epsilon d^*$}.
\end{assumption}
This assumption constrains the dynamics of $x$ with respect to $d$ (whether $d$ is above or below its threshold). Under this assumption $d<d^* \Rightarrow \dot{x} > 0$.

Under \cref{ass:dstar} it is easy to verify that the following transitions are unambiguous such that any trajectory in the first region will necessarily end up in the second region: RxD $\rightarrow$ Rxd, Rxd $\rightarrow$ rxd, rxd $\rightarrow$ rXd, rXd $\rightarrow$ rXD. Note that $1/\epsilon > d^* = \delta / \gamma$.

There are several possible cyclic sequences under \cref{ass:dstar} (see \cref{Fig4}). We decided to investigate the sequence
\begin{align}\label{eq:sequence}
\mbox{RxD $\rightarrow$ Rxd $\rightarrow$ rxd$\rightarrow$ rXd $\rightarrow$ rXD $\rightarrow$ RXD $\rightarrow$ RxD}
\end{align}
because of a numerical analysis, showing that this sequence is the most common one for random parameters. In other words, we find a large volume in parameter space that produces periodic orbits which follow this sequence.
In order to prove that trajectories follow sequence \eqref{eq:sequence}, we will provide conditions and proofs to avoid the dashed transitions in Fig. 4 and keep instead transitions rXD $\rightarrow$ RXD and RXD $\rightarrow$ RxD (labeled by "?" in \cref{Fig4}).

\subsection{Transition rXD to RXD}
In the region rXD, the dynamics are $\dx = 1-\epsilon d$ , $\dd = x - \beta d$ , $\dr = \gamma d - \delta r$, which allows for three transitions: rXD $\rightarrow$ rxD, rXD $\rightarrow$ RXD, rXD $\rightarrow$ rXd. We want to constrain the parameters in such a way, that only the transition rXD $\rightarrow$ RXD remains possible. Preliminaries, such as the upper and lower bounds on the initial conditions of $x$ are given in \cref{sec:appendix}.
\begin{lemma}\label{lem:tr}
Suppose stage rXD was reached from stage rXd. The time it takes for variable \(r\) to reach its threshold \(1\) is bounded from above by
\begin{equation*}
    T_r = \sqrt{\frac{2}{\gamma (\underline{x} - \beta /\epsilon)}} \ ,
\end{equation*}
where $\underline{x}$ is a lower bound on $x$ at the transition rXd $\rightarrow$ rXD (see \cref{lem:low_bound1}).
\end{lemma}
\begin{proof}
The trajectory has crossed the threshold \(d=d^*\) from below. Therefore, condition \(\dot{d} > 0\) holds at the transition, which implies \(x > \beta d^*\). \(\dot{x} > 1 - \epsilon d^* > 0\) holds by  \cref{ass:dstar}. Furthermore, \(\dot{r} > \gamma d^* - \delta r = \gamma \delta/\gamma - \delta r > 0 \) since \(r<1\). Consider
\begin{align*}
    \dot{x} &\geq 0 \Rightarrow x(t) \geq \underline{x} \\
    \dot{d} &\geq \underline{x} - \beta \sup d = \underline{x} - \beta /\epsilon \\ &\Rightarrow d(t) \geq (\underline{x} - \beta /\epsilon )t + d^*\\
    \dot{r} &\geq \gamma d - \delta \sup r \geq \gamma d - \delta \\ &\Rightarrow r(t) \geq \frac{\gamma}{2} (\underline{x} - \beta/\epsilon)t^2 + \underbrace{(\gamma d^* - \delta)}_{0}t + r_0 \ ,
\end{align*}
where $\underline{x}$ is a lower bound on the IC of $x$ (see \cref{lem:low_bound1}), \(r=1\) is $r$'s threshold, \(r_0 = 0\) is the lower bound on $r$'s initial condition, and \(d^*\) is the IC of \(d\). Denote \(\tau_r\) the time at which \(r=1\). Time \(\tau_r\) is then bounded from above by \(T_r\):
\begin{align*}
    r(\tau_r) &= 1 \geq \frac{\gamma}{2} (\underline{x} - \beta/\epsilon) \tau_r^2 \\
    \tau_r &\leq \sqrt{\frac{2}{\gamma(\underline{x} - \beta/\epsilon)}} =: T_r \ .
\end{align*}
\end{proof}

\begin{lemma}[threshold conditions \bs{rXD}]\label{lem:vrvd}
Suppose stage \bs{rXD} was reached from stage \bs{rXd}. The condition \(d=1/\epsilon\) must be fulfilled before either \(d\) or \(x\) reach their respective threshold.
\end{lemma}
\begin{proof}
The trajectory has crossed the threshold \(d=d^*\) from below. Since stages \bs{rXd} and \bs{rXD} are governed by the same equations, condition \(\dot{d} > 0\) holds on both sides of the transition, which implies \(x > \beta \delta/\gamma\). We also know, that \(\dot{x} > 1 - \epsilon d^* > 0\). In summary, \(\dot{d}>0, \dot{x}>0 \) at the time of threshold crossing. This holds true at least until \( d = 1/\epsilon\) at which point \( \dot{x} = 0\) and a subsequent decrease in \(x\) is followed by decreasing \(d\).
\end{proof}

\begin{lemma}\label{lem:td}
Suppose stage rXD was reached from stage  rXd. Suppose we stay in stage  rXD. The time it takes for variable \(d\) to reach \(1/\epsilon\) is bounded from below by
\begin{equation*}
    T_d = \frac{\sqrt{G^2 + 2H^2/\epsilon}-G}{H} \ ,
\end{equation*}
where \(G = \overline{x}  - \beta d^*\), \(H = 1 - \epsilon d^*\) and $\overline{x}$ is an upper bound on $x$ at the transition rXd $rightarrow$ rXD (see \cref{lem:up_bound1}).
\end{lemma}

\begin{proof}
    \begin{align*}
    \dot{x} &\leq 1 - \epsilon \inf d = 1 - \epsilon d^* \Rightarrow x(t) \leq (1 - \epsilon d^*)t + \overline{x}\\
    \dot{d} &\leq (1 - \epsilon d^*)t + \overline{x} - \beta \inf d  = (1 - \epsilon d^*)t + \overline{x} - \beta d^*\\
    &\Rightarrow d(t) \leq \frac{1}{2} (1 - \epsilon d^*)t^2 + (\overline{x} - \beta d^*)t + d^*\ ,
\end{align*}
where $\overline{x}$ of \cref{lem:up_bound1} was used as upper bound.
We get an expression for the time \(\tau_d\) at which \(d\) reaches \(1/\epsilon\):
\begin{align*}
    d(\tau_d) = 1/\epsilon \leq& \, \frac{1}{2} (1 - \epsilon d^*)\tau_d^2 + (\overline{x}  - \beta d^*)\tau_D + d^*
\end{align*}
which is a quadratic inequality in \(\tau_d\).
\begin{align*}
    0 \leq& \frac{1}{2} \underbrace{(1 - \epsilon d^*)}_{H}\tau_d^2 + \underbrace{(\overline{x}  - \beta d^*)}_{G}\tau_d -\underbrace{(1/\epsilon - d^*)}_{H/\epsilon}
\end{align*}
with \(H>0\) because of Assumption \ref{ass:dstar} and \(G>0\) from the transition condition on \(\overline{x}\) above. Since $\tau_d \geq 0$, there is only one solution:
\begin{equation*}
   \tau_d \geq \frac{\sqrt{G^2 + 2H^2/\epsilon}-G}{H} =: T_d\ .
\end{equation*}
\end{proof}

\begin{theorem}\label{th:transition1}
Suppose the stage \bs{rXD} was reached by stage \bs{rXd}. If the following inequality is fulfilled, the system can only transition to stage \bs{RXD}:
\begin{equation*}
    \sqrt{\frac{2}{\gamma (\underline{x} - \beta /\epsilon)}} < \frac{\sqrt{G^2 + 2H^2/\epsilon}-G}{H} \ .
\end{equation*}
\end{theorem}
\begin{proof}
By \cref{lem:tr}, the time it takes to reach \(r\)'s threshold is bounded from above by
\[T_r = \sqrt{\frac{2}{\gamma (\underline{x} - \beta /\epsilon)}} \ .\]
Lemma \ref{lem:vrvd} shows that the threshold \(d=1/\epsilon\) must be crossed before the system reaches either \(d\)'s or \(x\)'s threshold.
By \cref{lem:td}, the time it takes to reach \(d=1/\epsilon\) is bounded from below by
\[ T_d = \frac{\sqrt{G^2 + 2H^2/\epsilon}-G}{H} \ .\]
If \(T_r < T_d\), the only possible transition is \bs{rXD}  \(\rightarrow\) \bs{RXD}.
\end{proof}

\subsection{Transition RXD to RxD}
\begin{figure}[t]
    \centering
    \includegraphics[width=\columnwidth]{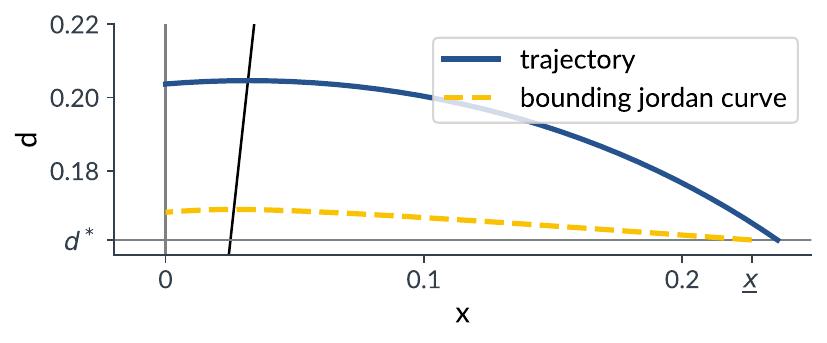}
    \caption{Sample trajectory and Jordan curve in region RXD. This is a projection of phase space onto the $(x,d)$-plane. A sample trajectory following the dynamics of subsytem RXD is shown in solid. In dashed, we show the Jordan curve described in \cref{lem:jordancurve}. The diagonal line separates the two subregions $\mathcal{A}$ and $\mathcal{B}$ of \cref{lem:jordancurve}.}
    \label{Fig5}
\end{figure}
In the region RXD, the dynamics are $\dx = -\epsilon d < 0, \dd = x - \beta d, \dr = \gamma d - \delta r$, which allows for two transitions: RXD $\rightarrow$ RXD, RXD $\rightarrow$ RxD. We want to constrain the parameters in such a way, that only the transition RXD $\rightarrow$ RxD remains possible.
The variable $r$ is above its threshold and cannot transition below during this stage, since $\dot{r}(r=1, d>d^*)>0$. Furthermore, variable $r$ does not influence the dynamics of the other variables. Therefore, we can safely consider the dynamics of the 2d-system of $x$, and $d$. In 2d, the proof is based on the construction of a so-called Jordan curve \cite[376]{munkres_topology_2000}, a closed curve which separates the plane into two components. In our case one component contains the trajectory and the other component contains the unwanted transition $d<d^*$ (see \cref{Fig5}). Since we are dealing with the region $d>=d^*, x\geq0$, there need only be a simple curve $\mathcal{J}$ connecting the two axes $x=0, d=d^*$ to create a separation.

\begin{lemma}[Jordan curve]\label{lem:jordancurve}
In region RXD, consider the 2d-system:
\begin{align*}
    \dot{x} &= - \epsilon d\\
    \dot{d} &= x - \beta d \ .
\end{align*}
There exists a curve $\mathcal{J}$ which separates any trajectory starting at $(x>\underline{x}, d=d^*)$ and the axis $(0\leq x < \underline{x}, d=d^*)$.
\end{lemma}
\begin{proof}
Consider two subregions: subregion $\mathcal{A}$ where $x>\beta d$ and subregion $\mathcal{B}$, where $x<\beta d$. In subregion $\mathcal{A}$, we create a bounding trajectory by using a new set of dynamics which has, for every coordinate ($x, d$), a flatter negative slope, than the original system . It can therefore never be crossed by the original trajectory. Let
\begin{align*}
    \dot{x}_m &= -\frac{\epsilon}{\beta} x_m\\
    \dot{d}_m &= x_m - \beta d_m
\end{align*}
which has the following slope
\begin{align*}
    \frac{\mathrm{d} \, d_m}{\mathrm{d}\, x_m} = \frac{x_m - \beta d_m}{-\frac{\epsilon}{\beta} x_m} \ .
\end{align*}
The slope of the original system is
\begin{align*}
    \frac{\mathrm{d} \, d}{\mathrm{d}\, x} = \frac{x - \beta d}{-\epsilon d} \ .
\end{align*}
Since $-d > -\frac{\epsilon}{\beta} x$ in subregion $\mathcal{A}$, it holds:
\begin{align*}
     \frac{\mathrm{d} \, d_m}{\mathrm{d}\, x_m} > \frac{\mathrm{d} \, d}{\mathrm{d}\, x} \ .
\end{align*}
In subregion $\mathcal{B}$, we create a bounding trajectory by using another set of dynamics which has, for every coordinate ($x, d$), a less steep positive slope than the original system. It can therefore never be crossed by the original trajectory. Let
\begin{align}
\begin{split}\label{eq:dp}
    \dot{x}_p &= -\epsilon d^*\\
    \dot{d}_p &= x_p - \beta d_p
\end{split}
\end{align}
which has the following slope:
\begin{align*}
    \frac{\mathrm{d} \, d_p}{\mathrm{d}\, x_p} = \frac{x_p - \beta d_p}{-\epsilon d^*} \ .
\end{align*}
Since $-d < -d^*$, it holds:
\begin{align*}
     \frac{\mathrm{d} \, d_p}{\mathrm{d}\, x_p} < \frac{\mathrm{d} \, d}{\mathrm{d}\, x} \ .
\end{align*}
Let $\mathcal{J}$ be the curve that is described by the trajectory starting at IC $(\underline{x}, d^*)$, and following dynamics of subregion $\mathcal{A}$ until $x = \beta d$ at time $t_1$ and following dynamics of subregion $\mathcal{B}$ until $x=0$ at time $t_2$. By the arguments above $\mathcal{J}$ separates the phase space into two disconnected components, one containing all trajectories of the real system and the other one containing $(0\leq x < \underline{x}, d=d^*)$.
\end{proof}

\begin{theorem}[final condition]\label{th:curve_end}
Under the following conditions, the bounding Jordan curve $\mathcal{J}$ will terminate at a point $O = (x=0, d>d^*)$:
\begin{align}\label{eq:cond2}
x_m(t_1) \left( \frac{1}{\beta} - \frac{1}{2 \epsilon d^*} x_m(t_1) \right) &> d^*
\end{align}
where
\begin{align*}
x_m(t_1) &= \underline{x} \exp\left(-\frac{\epsilon}{\beta^2 - \epsilon}\ln\left( \beta\frac{d^* + F}{\underline{x} + \beta F}\right) \right) \\
F &= \frac{\beta \underline{x}}{\epsilon - \beta^2} \ .
\end{align*}
\end{theorem}

\begin{proof}
The curve $\mathcal{J}$ is described in \autoref{lem:jordancurve}. It is produced by the trajectory of the system:
\begin{align*}
  \left\lbrace
  \begin{array}{l@{}l}
    (x_m, d_m) & \ \text{if }\ (x, d) \in \mathcal{A}\\
    (x_p, d_p) & \ \text{if }\ (x, d) \in \mathcal{B}
  \end{array}
  \right.
\end{align*}
with
\begin{align*}
\dot{x}_m &= -\frac{\epsilon}{\beta} x_m\\
\dot{d}_m &= x_m - \beta d_m\\
\dot{x}_p &= -\epsilon d^*\\
\dot{d}_p &= x_p - \beta d_p \ .
\end{align*}
First, We will integrate the trajectory starting from $(\underline{x}, d^*)$ and reaching $x=\beta d$.
\begin{align*}
x_m(t) &= \underline{x} \exp\left(-\frac{\epsilon}{\beta}t\right)\\
d_m(t) &= -F \exp\left(-\frac{\epsilon}{\beta}t \right) + (d^*+F) \exp(-\beta t)\\
\text{where} \ F &= \frac{\beta \underline{x}}{\epsilon - \beta^2} \ .
\end{align*}
Using the ending condition $x(t_1)=\beta d(t_1)$:
\begin{align*}
\left( F + \frac{\underline{x}}{\beta} \right) \exp\left(-\frac{\epsilon}{\beta}t_1 \right) &= (d^* + F) \exp(-\beta t_1)\\
t_1 = \frac{1}{\beta - \epsilon/\beta}&\ln\left(\frac{d^* + F}{\underline{x}/\beta + F}\right) \ .
\end{align*}
Second, we will integrate the trajectory starting from $(x_p(0)=x_m(t_1), d_p(0)=d_m(t_1))$ and reaching \mbox{$x_p(t_2)=0$}:
\begin{align}
x_p(t) &= x_m(t_1) - \epsilon d^* t \nonumber \\
d_p(t) &= b_0 + b_1t + b_2\exp(-\beta t) \\
\text{where} \ b_0 &= \frac{x_m(t_1)}{\beta} + \frac{\epsilon d^*}{\beta^2} \nonumber\\
b_1 &= \frac{\epsilon d^*}{\beta} \nonumber\\
b_2 &= -\frac{\epsilon d^*}{\beta^2} \ . \nonumber
\end{align}
The final condition yields:
\begin{align}
x_p(t_2) &= x_m(t_1) - \epsilon d^* t_2 = 0 \nonumber\\
\Rightarrow t_2 &= \frac{x_m(t_1)}{\epsilon d^*} \ . \label{eq:t2}
\end{align}
Last, we need to show that $d_p(t_2) > d^*$. Keeping in mind that $\exp(-x) < (1-x+x^2/2) \quad \forall x>0$ and $b_2<0$, we can calculate a lower bound:
\begin{align*}
d_p(t_2) &= b_0 + b_1 t_2 + b_2 \exp(-\beta t_2) \\
 &> b_0 + b_1 t_2 + b_2 \left(1-\beta t_2+\frac{\beta^2 t_2^2}{2}\right)\\
 &=  \frac{x_m(t_1)}{\beta} + t_2(b_1 - \beta b_2) + t_2^2 \left(\frac{1}{2}\beta^2 b_2 \right) \\
&>  \frac{x_m(t_1)}{\beta} - \frac{1}{2} \epsilon d^* t_2^2\\
&= x_m(t_1) \left(\frac{1}{\beta} - \frac{x_m(t_1)}{2\epsilon d^*}\right)
\end{align*}
where
\begin{align*}
x_m(t_1) &= \underline{x} \exp\left(-\frac{\epsilon}{\beta^2 - \epsilon}\ln\left( \beta\frac{d^* + F}{\underline{x} + \beta F}\right) \right) \ .
\end{align*}
Now, if
\begin{align} \label{eq:jc}
x_m(t_1) \left(\frac{1}{\beta} - \frac{x_m(t_1)}{2\epsilon d^*}\right) > d^*
\end{align}
it follows that $d_p(t_2) > d^*$. Therefore if a set of parameters fulfills \eqref{eq:jc}, the Jordan curve will terminate at $O = (x_p=0, d_p(t_2)>d^*)$.
\end{proof}

\begin{figure}[t!]
    \centering
    \includegraphics[width=\columnwidth]{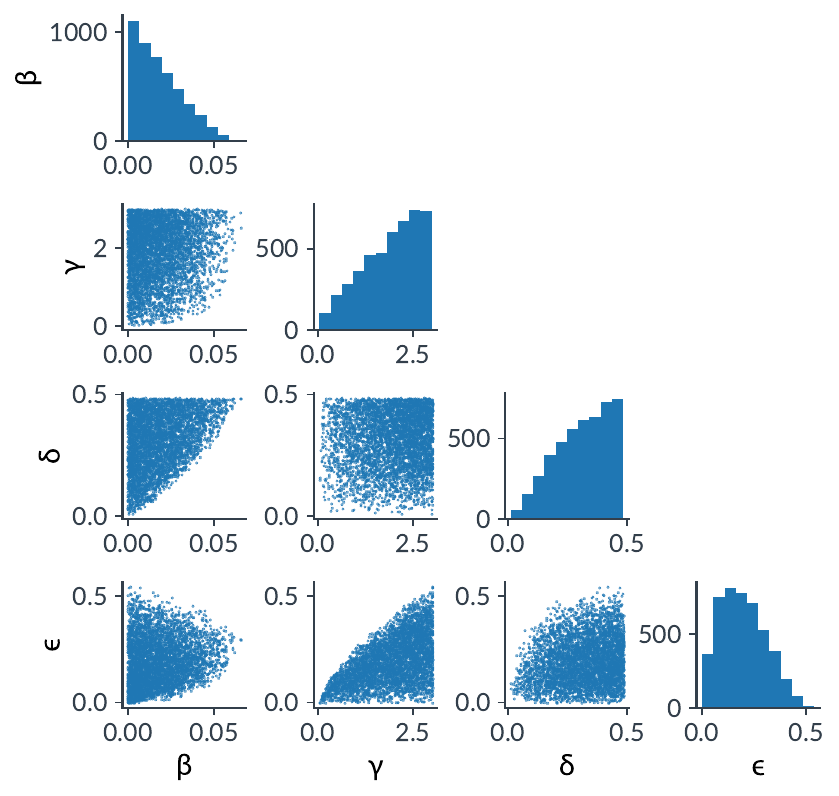}
    \caption{Region in parameter space admitting oscillations. Shown are 2d projections of individual parameter sets fulfilling the sufficient conditions for the given cyclic sequence of subsystems. Histograms of the number of valid values per parameter are shown on the diagonal.}
    \label{Fig6}
\end{figure}

\begin{remark}
If the Jordan curve intersects $x=0$ with $d>d^*$, a transition of variable $d$ is impossible and only the threshold at $x=0$ can be attained.
\end{remark}

Recall the definitions of $t_2$ in \cref{lem:jordancurve} and $d_p(t_2)$ in \cref{eq:dp}.

\begin{corollary}
Under \cref{ass:dstar}, \cref{th:transition1} and \cref{th:curve_end}, the only possible transition of subsytems is the cyclic one given by: RxD $\rightarrow$ Rxd $\rightarrow$ rxd $\rightarrow$ rXd $\rightarrow$ rXD $\rightarrow$ RXD$\rightarrow$ RxD. If all solutions to \cref{eq:pwa} remain bounded (as they do in our numerical simulations), using the $(r,d\geq d_p(t_2))$-plane at the transition RXD$\rightarrow$RxD (which does not include the point $(r^*, x^*, d^*)$) by the Brouwer fixed point theorem \cite[51]{zeidler_nonlinear_1985}, there exists a periodic orbit.
\end{corollary}
By sampling the parameter space, we find a dense region which fulfills all conditions: \cref{ass:dstar}, \cref{th:transition1} and \cref{th:curve_end} (see Fig. \ref{Fig6}).

In this section we showed the existence of a periodic orbit in our piecewise affine model for a fairly large parameter region. Now, that we established the existence of a periodic orbit, we will study its responses to external inputs.

\section{Dynamics under external inputs}\label{sec:prc}
\begin{figure}[t]
    \centering
    \includegraphics[width=\columnwidth]{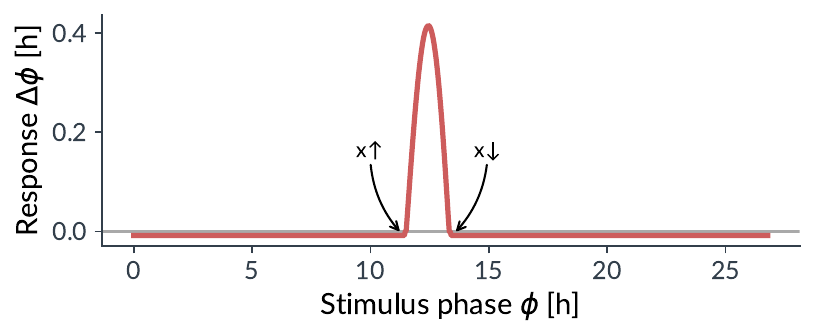}
    \caption{Phase Response-Curve of the PWA System \eqref{eq:pwa}. The phase responses were obtained by direct numerical integration. Negative stimuli were applied on variable $x$ as proxy for positive stimuli on $p$. Stimulus amplitude $A = -0.5$, stimulus length $L= 0.05$\,h. The upward and downward transition times of variable $x$ are indicated. Parameters as in \cref{Fig3}.}
    \label{Fig7}
\end{figure}
For a deeper understanding of the system, we generated the Phase Response-Curve (PRC) of our piecewise affine model. The PRC is an input-output measure that tabulates the measured phase responses against the timing of a localized input. Phase responses can be phase advances or phase delays of the oscillator leading to a period shortening or lengthening, respectively.

We constructed the Phase Response-Curve using direct numerical integration (see \cref{Fig7}). A pulse was applied to variable $x$ on 100 different phases in the cycle. The pulse is a square pulse of length $L=0.05$\,h and amplitude $A=-0.5$. A negative pulse on $x$ was used, since this is equivalent to a positive perturbation of variable $p$ ($x = b-p$, see \cref{sec:seqrate}).

Well-known PRCs of canonical models close to bifurcations that produce periodic orbits follow one of two forms depending on their classification. Class I oscillators follow a $\Delta \phi(\phi) = 1-\cos(\phi)$ form and Class II oscillators a $\Delta \phi(\phi) = -\sin(\phi)$ form \cite{sacre_sensitivity_2014}. Remarkably, the PRC of our model is close to zero almost everywhere. This is a so-called \textit{dead-zone} \cite[171]{forger_biological_2017}. It indicates, that the oscillation is very insensitive to external perturbations during a time window which comprises the greater part of the cycle.

The stimulus phases, where the response is non-zero, lie exactly in the region where $x>0$ (transitions of $x$ shown in \cref{Fig7}). To understand the form of this PRC, note that in the subsystems where $x<0$, the variable $x$ has no influence on other variables or even itself, since $f_b(x) = 0$. Any perturbation will therefore only change the state of $x$, but will not propagate through the system.
As a consequence, while $x<0$, an external perturbation applied on variable $x$, does not produce differential responses, but a constant response and thus a flat PRC.

Variable $x$ in model \eqref{eq:pwa} is the difference between the concentrations of the two complexes $b$ and $p$ that mutually inactivate each other. $x>0$ corresponds to the interval where the complex CLOCK:BMAL1 is not repressed and can trigger the activation of all clock-controlled genes.

In terms of concentrations, the near zero responses are a result of the fact that external perturbations on [PER:CRY] do not have an influence on the concentrations [DBP] and [REV] while CLOCK:BMAL1 is repressed.
This is a consequence of the fast repression by sequestration. Conversely, perturbations on [PER:CRY] have an immediate influence on [DBP] and [REV] when CLOCK:BMAL is not repressed. Every increase in [PER:CRY] is immediately followed by the mutual sequestration of PER:CRY and CLOCK:BMAL1 and thus transformed into a negative perturbation in [CLOCK:BMAL1]. It is therefore the sequestration process which interlocks the dynamics of [CLOCK:BMAL1] and [PER:CRY], that is at the heart of this lack of delays as phase response.
Note that in the limiting case $\alpha \rightarrow \infty$ sequestration is instantaneous. As explained in \cref{sec:seqrate}, sequestration is in general faster than other processes.

\subsection{Comparison with Goodwin oscillator}
\begin{figure}[t]
    \centering
    \includegraphics[width=\columnwidth]{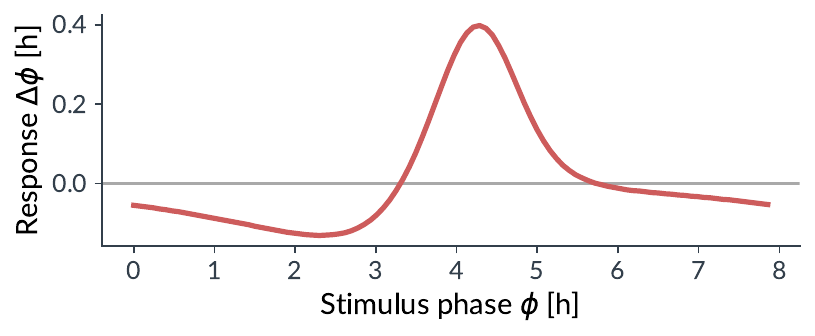}
    \caption{PRC of the Goodwin oscillator (\cref{eq:goodwin}). A single stimulus was applied at 100 phases to the $z$ variable of the system. Stimulus amplitude $A = 0.5$, stimulus length $L = 0.05$\,h. Parameters: $\alpha_1 = \alpha_2 = \alpha_3 = 5.0, \gamma_1 = \gamma_2 = \gamma_3 = 0.5, K = 1.0,n = 10.0$.}
    \label{Fig8}
\end{figure}

As a comparison, we also generated the PRC for a periodic system with no sequestration term, the Goodwin oscillator (see \cref{Fig8}). We assume external perturbations to influence the protein concentration and thus variable $Z$. Model equations and parameters are taken from Fig. 2 in \cite{gonze_goodwin_2021}. The equations are as follows:
\begin{align}
\begin{split}\label{eq:goodwin}
    \dot{X} &= \alpha_1 \frac{K^n}{K^n + Z^n} - \gamma_1 X\\
    \dot{Y} &= \alpha_2 X - \gamma_2 Y\\
    \dot{Z} &= \alpha_3 Y - \gamma_3 Z \ .
\end{split}
\end{align}
The PRC of the Goodwin oscillator was constructed with a square pulse of length $L=0.05$\,h and amplitude $A=0.5$.
Perturbations on variable $Z$ trigger a delay in one part of the cycle and an advance in another part, like a class II oscillator.
The Phase Response-Curve of our oscillator, on the other hand, shows negligible phase delays for a big part of the cycle and comparatively large phase advances in a small time window, restricted by $x$.
For a more detailed interpretation, see the Discussion.

\section{Entrainment Properties of Cellular Oscillators}\label{sec:entrainment}
\begin{figure}[!t]
    \centering
    \includegraphics[width=\columnwidth]{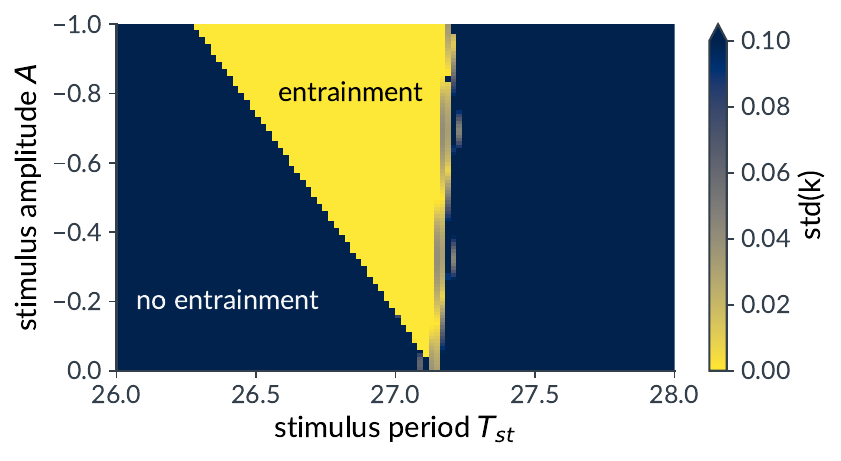}
    \caption{Arnold tongue of the piecewise-affine model. A periodic square-wave stimulus with stimulus length = 0.05\,h was added to the $x$-variable of the system. Amplitude $A$ and stimulus period $T_{\text{st}}$ are varied. Color coded is the standard deviation of the Kuramoto order parameter $k$ calculated from the phases of variable $x$ and the stimulus. Parameters as in \cref{Fig3}.}
    \label{Fig9}
\end{figure}

\noindent Whether a periodic stimulus can entrain an oscillator depends on the difference of stimulus period $T_{\text{st}}$ and freerunning period of the oscillator $T_{\text{fr}}$ and also on the stimulus amplitude. In the range of weak perturbations, an external stimulus only triggers a phase shift in the perturbed oscillator. For this case, a stable phase relationship between an oscillator of period $T_{\text{fr}}$ and a periodic stimulus of period $T_{\text{st}}$ can only be maintained if the period difference can be compensated by the phase shift induced by the external stimulus.

Therefore, the Phase Response-Curve, which quantifies the possible phase shifts of a single stimulus, is able to qualitatively and quantitatively predict the range of entrainment.

We use Arnold tongues to visualize the entrainment region of an oscillator to a specific periodic stimulus (see \cite{boccaletti_synchronization_2018} for an introduction). With increasing stimulus amplitude, a greater period difference of freerunning and stimulus period can be compensated. We quantify the entrainment of an oscillator via the Kuramoto order paramter. The Kuramoto order parameter $k(t)$ is the amplitude of the complex mean of all phases in a network of N coupled oscillators:
\begin{equation*}
k(t) = \left| \frac{1}{N} \sum_{j}^{N} e^{i\phi_j(t)} \right| \ .
\end{equation*}
$k$ stays constant if the phase difference between the concerned phases $\phi_j$ does not change over time. We use the standard deviation std($k$) after a transient as the numerical measure for entrainment. For the PWA system, we simulated the oscillator for $25000$\,h and analyzed the last 25\,\% of the time series. We transformed the timeseries of the stimulus and the oscillatior into phase variables. If std$(k) \approx$ 0, $k$ stays constant and the periodic stimulus and the oscillator have a stable phase relationship.
\begin{figure}[!t]
    \centering
    \includegraphics[width=\columnwidth]{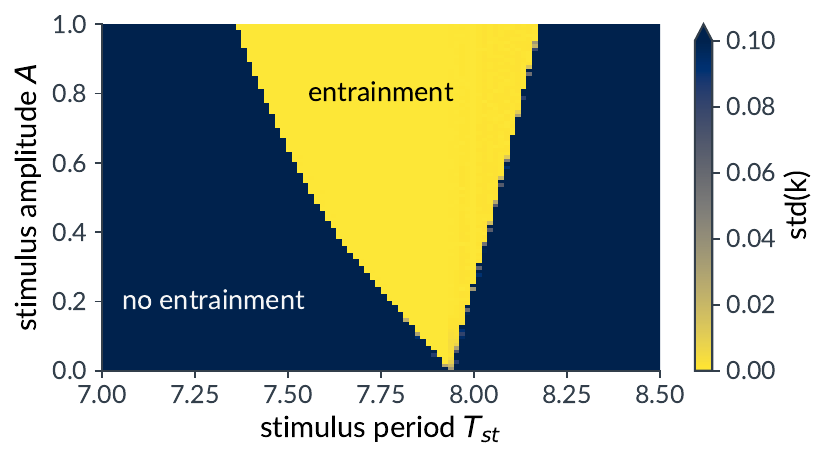}
    \caption{Arnold tongue of Goodwin model. A periodic square-wave stimulus with stimulus length = $0.05 h$ was added to the $Z$-variable of the system. Amplitude $A$ and stimulus period $T_{\text{st}}$ are varied. Color coded is the standard deviation of the Kuramoto order parameter $k$ calculated from the phases of variable $x$ and the stimulus. Parameters as in \cref{Fig8}}
    \label{Fig10}
\end{figure}
The stimulus we use, is the same used to generate the Phase Response-Curves in \cref{sec:prc}, but in a periodic fashion. It is a square pulse of fixed length of $L=0.05$\,h with amplitude $A$. The stimulus period $T_{\text{st}}$ determines the distance of two successive square pulses.

In \cref{Fig9}, the entrainment region of the PWA-oscillator is shown. It exhibits a well-entrained region for stimulus periods which are smaller than the freerunning period of $T_{\text{fr}} \approx 27.2$ ($T_{\text{st}} < T_{\text{fr}}$), but a much smaller entrainment region for $T_{\text{st}} > T_{\text{fr}}$.

This behaviour is reflected in the Phase Response-Curve of the PWA oscillator. Note, that the stimulus amplitude used for the PRC is $A=-0.5$ and the the maximal phase advance (period shortening) was 0.4\, h. This corresponds to the minimal stimulus period that entrains the oscillator in \cref{Fig9}, which is $\approx 0.4$\, h shorter than the freerunning period.

A single stimulus can advance the phase of the oscillator (and thus shorten its period) if it arrives in the window $x>0$. But the possible phase delays are much smaller, leading only to a marginal lengthening of the oscillator period.

As we have described in \cref{sec:prc}, this structure of large phase advances but marginal phase delays in the PRC is a direct result of the structure of system \eqref{eq:pwa}.

The largest possible phase delay translates directly into the largest $T_{\text{st}}$ that is able to entrain the oscillator.
It is therefore the direct consequence of the interlocking of variables $b$ and $p$ due to the sequestration process, that gives rise to this one-sided Arnold's tongue.

To compare, \cref{Fig10} shows the Arnold's tongue of the Goodwin oscillator. Here we simulated the oscillator for $5000$\,h and analyzed the last 25\,\% of the time series. It exhibits entrainment for stimuli with $T_{\text{st}}< T_{\text{fr}}$ and $T_{\text{st}} > T_{\text{fr}}$. Again, this is due to the structure of the phase response-curve, which shows that a square wave stimulus can advance or delay the phase of the Goodwin oscillator, depending on the phase relationship between the two. Consequently, the Goodwin oscillator can be entrained by periodic stimuli with shorter but also with longer periods than its freerunning period.

Comparison of the entrainment properties for the two oscillators suggests that the role of the sequestration motif may be to restrict the range of responses of the system to external perturbations, and thus introduce a form of robustness of the oscillator by averting deviations from the original period.

\section{Discussion}
We study a model of the major transcriptional and post-translational processes in mammalian circadian clocks.

The model describes the interacting dynamics of core clock proteins DBP, REV-ERB and protein complexes CLOCK:BMAL1, PER:CRY. The sequestering of complexes CLOCK:BMAL1 and PER:CRY is a major post-translational process in the core circadian clock network. It has been modeled in different fashions in the past. In \cite{hesse_mathematical_2021} the process is not modeled directly but represented as multiplicative competition in the synthesis terms in which BMAL1 plays an activating role. \textcite{kim_mechanism_2012} model the sequestration as a static process starting from the laws of mass action with the additional assumption of constant activator concentration. \textcite{francois_core_2005} use mass-action dynamics for their description of an oscillator with one feeback loop. In addition, they consider an irreversible sequestration, so no dissociation is taken into account.
We use an oscillator with two feedback loops to describe the core circadian clock, where the sequestration is modelled dynamically, like in \cite{francois_core_2005}.

We show that the mass-action process describing the sequestration separates two timescales in the model, when appropriately transformed. The timescale separation is a result of the tightly interlocked dynamics of the sequestering species. Because of the interlocking, it is the difference of concentrations [CLOCK:BMAL1] and [PER:CRY] that moves on the slow timescale, together with the dynamics of [DBP] and [PER].
We use this timescale separation to show that our system can be easily reduced and transformed into a 3 dimensional piecewise affine system with linear time-invariant subsystems.

By systematically excluding specific transitions between the subsystems, we show that under certain conditions, only one sequence of periodic transitions remains possible. Using the Brouwer fixed point theorem, we then prove the existence of a periodic orbit. Furthermore, we show numerically that all posed conditions are fulfilled for a rather large region in parameter space. \\
Next, the Phase Response-Curve for the system is studied.
It depends highly on the state of the aforementioned difference [CLOCK:BMAL1] - [PER:CRY].
It shows a characteristic shape, being almost zero almost everywhere and positive in the region where the difference variable is positive. In terms of concentrations, the oscillator is, paradoxically, almost insensitive to stimuli in [PER:CRY] when [PER:CRY] is nonzero. It is, however, very sensitive to stimuli, when [PER:CRY] is zero. This is due to the fast sequestering of any surplus [PER:CRY], which decreases [CLOCK:BMAL1] and has downstream effects in [DBP] and [REV].

In contrast, perturbation on the inhibiting factor in a model with strongly cooperative Hill-type repression (Goodwin oscillator), produces both phase advances and phase delays.
We exemplify this notion further by showing the regions of entrainment for two different oscillators, our sequestering-type repression oscillator and the Goodwin oscillator with Hill-type repression.
The Goodwin oscillator can be entrained by a periodic stimulus which has a longer period than the freerunning oscillator period. The periodic stimulus can force the Goodwin oscillator to share the stimulus period and have a stable phase relationship to the stimulus, given that its amplitude is high enough.
The sequestration type oscillator, on the other hand, cannot be entrained by a stimulus with a period significantly higher than its own freerunning period. The sequestration motif hinders the entrainment of longer period stimuli.

There is to our knowledge no experimental evidence for a PRC with large constant zone and rather tight window of nontrivial phase responses. Usually, experimental phase response curves show zones of phase advances and phase delays \cite{sharmad_light-induced_1996}. However, we show that if the repressor in a sequestration-based oscillator is subjected to external input, the tight binding of repressor to activator should lead to phase-response curves of this kind.

In the context of biological clocks, this article investigates thoroughly the sequestration repression motif. We have shown that the dynamical modelling of the sequestration process introduces a symmetry in the concentration dynamics. Because of the nonlinearity of the mass-action process, a timescale separation independent of the sequestration rate is possible after change of variables. Furthermore, we use this separation to create a piecewise affine system and prove the existence of a periodic orbit. Finally, we show that the tight interlocking of the activator and repressor species gives rise to a unique phase-response curve. Altogether, our results highlight the unique dynamics due to sequestration repression compared to highly cooperative repression types.

\section*{Acknowledgments} This work was supported in part by the French National Agency for Research through project InSync ANR-22-CE45-0012-01.

\printbibliography

\begin{appendix}
\section{Proofs for cyclic transitions}\label{sec:appendix}
\subsection{Preliminaries}\label{sec:pre}
\begin{lemma}[upper bound on \(x\) after stage rXd]\label{lem:up_bound1}
Suppose that the previous stage was rxd. Suppose that the stage rXd finishes by transitioning to stage rXD. The maximal value variable \(x\) can assume is
\begin{equation*}
    \overline{x} := \frac{\sqrt{\beta^2 d^{*2} + 2Hd^*} - \beta d^*}{H} ,
\end{equation*}
where \(H = 1 - \epsilon d^*\).
\end{lemma}
\begin{proof}
\(x\) crosses 0 during the transition rxd \(\rightarrow\) rXd. We can therefore calculate an upper bound on \(x\) at the end of stage rXd. In order to do this, we need to find an upper bound on the time it takes for \(d\) to reach its threshold since this marks the end of the stage.
\begin{align*}
    \dot{x} &\geq 1 - \epsilon \sup d = 1 - \epsilon d^* \\
    & \Rightarrow x(t) \geq (1 - \epsilon d^*)t \\\\
    \dot{d} &\geq x - \beta \sup d = (1 - \epsilon d^*)t - \beta d^* \\
    &\Rightarrow d(t) \geq \frac{1}{2} (1  - \epsilon d^*)t^2 - \beta d^* t
\end{align*}
We used \(\sup d = d^*\) because we are in stage rXd. Inserting the time \(\tau_d\) when \(d\) reaches its threshold \(d^*\), we get an upper bound for it:
\begin{align*}
    d(\tau_d) &= d^* \geq \frac{1}{2} (1 - \epsilon d^*)\tau_d^2 - \beta d^* \tau_d\\
    0 &\geq \frac{1}{2} \underbrace{(1 - \epsilon d^*)}_{H}\tau_d^2 - \beta d^*\tau_d - d^*\\
    \tau_d &\leq \frac{\sqrt{\beta^2 d^{*2} + 2Hd^*} - \beta d^*}{H} \ .
\end{align*}
Note, that \(\dot{x} \leq 1 - \epsilon \inf d = 1 \).
This can be used to calculate a maximal value of \(d\) at the end of the stage:
\begin{equation*}
    x(\tau_d) \leq \tau_d \leq \frac{\sqrt{\beta^2 d^{*2} + 2Hd^*} - \beta d^*}{H} =: \overline{x} \ .
\end{equation*}
\end{proof}

\begin{lemma}[lower bound on \(x\) after stage rXd]\label{lem:low_bound1}
Suppose that the previous stage was \bs{rxd}. Suppose that the stage \bs{rXd} finishes by transitioning to stage \bs{rXD}. The minimal value variable \(x\) can assume is
\begin{equation*}
    \underline{x} :=  H \sqrt{d^*} \ ,
\end{equation*}
where \(H = 1 - \epsilon d^*\).
\end{lemma}
\begin{proof}
\(x\) crosses 0 during the transition \bs{rxd} \(\rightarrow\) \bs{rXd}. We can therefore calculate a lower bound on \(x\) at the end of stage \bs{rXd}. In order to do this, we need to find a lower bound on the time it takes for \(d\) to reach its threshold since this marks the end of the stage.
\begin{align*}
    \dot{x} &\leq 1 - \epsilon \inf d = 1 \\
    & \Rightarrow x(t) \leq t \\\\
    \dot{d} &\leq x - \beta \inf d = t \\
    &\Rightarrow d(t) \leq \frac{1}{2} t^2
\end{align*}
Using the time \(\tau_d\) when \(d\) reaches its threshold \(d^*\):
\begin{align*}
    d(\tau_d) &= d^* \leq \frac{1}{2} \tau_d^2\\
    \tau_d &\geq \sqrt{d^*} \ .
\end{align*}
Note, that \(\dot{x} \geq 1 - \epsilon \sup d = 1 - \epsilon d^*\).
This can be used to calculate a lower bound of \(x\) at the end of the stage:
\begin{equation*}
    x(\tau_d) \geq (1 - \epsilon d^*) \tau_d \geq (1 - \epsilon d^*) \sqrt{d^*} =: \underline{x} \ .
\end{equation*}
\end{proof}

\begin{assumption} \label{bounds_check}
    Assume $\overline{x} > \underline{x}$.
\end{assumption}

\begin{remark}
\cref{bounds_check} implies that the upper bound on \(x\) is greater than the lower bound on \(x\) in stage \bs{rXD}, since \(\overline{x}' > \overline{x}\). The bounds on X are therefore \(\underline{x} \leq x \leq \overline{x}\) after stage \bs{rXd} and \(\underline{x} \leq x \leq \overline{x}'\) after stage \bs{rXD}.
\end{remark}

\subsection{Transition rXD to RXD}
\begin{lemma}[threshold conditions \bs{rXD}]\label{lem:vrvd_app}
Suppose stage \bs{rXD} was reached from stage \bs{rXd}. The condition \(d=1/\epsilon\) must be fulfilled before either \(d\) or \(x\) reach their respective threshold.
\end{lemma}
\begin{proof}
The trajectory has crossed the threshold \(d=d^*\) from below. Since stages \bs{rXd} and \bs{rXD} are governed by the same equations, condition \(\dot{d} > 0\) holds on both sides of the transition, which implies \(x > \beta \delta/\gamma\). We also know, that \(\dot{x} > 1 - \epsilon d^* > 0\). In summary, \(\dot{d}>0, \dot{x}>0 \) at the time of threshold crossing. This holds true at least until \( d = 1/\epsilon\) at which point \( \dot{x} = 0\) and a subsequent decrease in \(x\) is followed by decreasing \(d\).
\end{proof}

\begin{theorem}
Suppose the stage \bs{rXD} was reached by stage \bs{rXd}. If the following inequality is fulfilled, the system can only transition to stage \bs{RXD}:
\begin{equation*}
    \sqrt{\frac{2}{\gamma (\underline{x} - \beta /\epsilon)}} < \frac{\sqrt{G^2 + 2H^2/\epsilon}-G}{H} \ .
\end{equation*}
\end{theorem}
\begin{proof}
By lemma \ref{lem:tr}, the time it takes to reach \(r\)'s threshold is bounded from above by
\[T_r = \sqrt{\frac{2}{\gamma (\underline{x} - \beta /\epsilon)}} \ .\]
Lemma \ref{lem:vrvd} shows that the threshold \(d=1/\epsilon\) must be crossed before the system reaches either \(d\)'s or \(x\)'s threshold.
By lemma \ref{lem:td}, the time it takes to reach \(d=1/\epsilon\) is bounded from below by
\[ T_d = \frac{\sqrt{G^2 + 2H^2/\epsilon}-G}{H} \ .\]
If \(T_r < T_d\), the only possible transition is \bs{rXD}  \(\rightarrow\) \bs{RXD}.
\end{proof}

\subsection{Transition RXD to RxD}
\begin{lemma}[initial condition]\label{lem:ic}
Suppose that a trajectory followed the stage progression \bs{rXd, RXD} and is now in stage \bs{RXD}.The trajectory traversing stage \bs{RXD} is bounded from below by the trajectory starting with the initial condition $x=\underline{x}, d=d^*$.
\end{lemma}
\begin{proof}
In stage \bs{rXD}, the following holds by \autoref{lem:vrvd}: \(\dot{x}>0, \dot{d}>0\). Any initial condition in \bs{RXD}, will therefore have an initial condition $(x_{IC}, d_{IC}) \geq (\underline{x}, d^*)$, where equality is attained only if $T_r = 0$. Since the dynamics are linear and trajectories in 2-d cannot cross because of uniqueness, the real trajectory is bounded from below by the one starting from $(\underline{x}, d^*)$.
\end{proof}

\begin{theorem}[final condition]\label{lem:curve_end_app}
Under the following conditions, the bounding Jordan curve $\mathcal{J}$ will terminate at a point $O = (x=0, d>d^*)$:
\begin{align}\label{eq:cond2_app}
x_m(t_1) \left( \frac{1}{\beta} - \frac{1}{2 \epsilon d^*} x_m(t_1) \right) &> d^*
\end{align}
where
\begin{align*}
x_m(t_1) &= \underline{x} \exp\left(-\frac{\epsilon}{\beta^2 - \epsilon}\ln\left( \beta\frac{d^* + F}{\underline{x} + \beta F}\right) \right) \\
F &= \frac{\beta \underline{x}}{\epsilon - \beta^2} \ .
\end{align*}
\end{theorem}

\begin{proof}
The curve $\mathcal{J}$ is described in \autoref{lem:jordancurve}. It is produced by the trajectory of the system:
\begin{align*}
  \left\lbrace
  \begin{array}{l@{}l}
    (x_m, d_m) & \ \text{if }\ (x, d) \in \mathcal{A}\\
    (x_p, d_p) & \ \text{if }\ (x, d) \in \mathcal{B}
  \end{array}
  \right.
\end{align*}
with
\begin{align*}
\dot{x}_m &= -\frac{\epsilon}{\beta} x_m\\
\dot{d}_m &= x_m - \beta d_m\\
\dot{x}_p &= -\epsilon d^*\\
\dot{d}_p &= x_p - \beta d_p \ .
\end{align*}
First, We will integrate the trajectory starting from $(\underline{x}, d^*)$ and reaching $x=\beta d$.
\begin{align*}
x_m(t) &= \underline{x} \exp\left(-\frac{\epsilon}{\beta}t\right)\\
d_m(t) &= -F \exp\left(-\frac{\epsilon}{\beta}t \right) + (d^*+F) \exp(-\beta t)\\
\text{where} \ F &= \frac{\beta \underline{x}}{\epsilon - \beta^2} \ .
\end{align*}
Using the ending condition $x(t_1)=\beta d(t_1)$:
\begin{align*}
\left( F + \frac{\underline{x}}{\beta} \right) \exp\left(-\frac{\epsilon}{\beta}t_1 \right) &= (d^* + F) \exp(-\beta t_1)\\
t_1 = \frac{1}{\beta - \epsilon/\beta}&\ln\left(\frac{d^* + F}{\underline{x}/\beta + F}\right) \ .
\end{align*}
Second, we will integrate the trajectory starting from $(x_p(0)=x_m(t_1), d_p(0)=d_m(t_1))$ and reaching \mbox{$x_p(t_2)=0$}:
\begin{align}
x_p(t) &= x_m(t_1) - \epsilon d^* t \nonumber \\
d_p(t) &= b_0 + b_1t + b_2\exp(-\beta t) \\
\text{where} \ b_0 &= \frac{x_m(t_1)}{\beta} + \frac{\epsilon d^*}{\beta^2} \nonumber\\
b_1 &= \frac{\epsilon d^*}{\beta} \nonumber\\
b_2 &= -\frac{\epsilon d^*}{\beta^2} \ . \nonumber
\end{align}
The final condition yields:
\begin{align}
x_p(t_2) &= x_m(t_1) - \epsilon d^* t_2 = 0 \nonumber\\
\Rightarrow t_2 &= \frac{x_m(t_1)}{\epsilon d^*} \ . \label{eq:t2_app}
\end{align}
Last, we need to show that $d_p(t_2) > d^*$. Keeping in mind that $\exp(-x) < (1-x+x^2/2) \quad \forall x>0$ and $b_2<0$, we can calculate a lower bound:
\begin{align*}
d_p(t_2) &= b_0 + b_1 t_2 + b_2 \exp(-\beta t_2) \\
 &> b_0 + b_1 t_2 + b_2 \left(1-\beta t_2+\frac{\beta^2 t_2^2}{2}\right)\\
 &=  \frac{x_m(t_1)}{\beta} + t_2(b_1 - \beta b_2) + t_2^2 \left(\frac{1}{2}\beta^2 b_2 \right) \\
&>  \frac{x_m(t_1)}{\beta} - \frac{1}{2} \epsilon d^* t_2^2\\
&= x_m(t_1) \left(\frac{1}{\beta} - \frac{x_m(t_1)}{2\epsilon d^*}\right)
\end{align*}

where
\begin{align*}
x_m(t_1) &= \underline{x} \exp\left(-\frac{\epsilon}{\beta^2 - \epsilon}\ln\left( \beta\frac{d^* + F}{\underline{x} + \beta F}\right) \right) \ .
\end{align*}
Now, if
\begin{align} \label{eq:jc2}
x_m(t_1) \left(\frac{1}{\beta} - \frac{x_m(t_1)}{2\epsilon d^*}\right) > d^*
\end{align}
it follows that $d_p(t_2) > d^*$. Therefore if a set of parameters fulfills \eqref{eq:jc2}, the jordan curve will terminate at $O = (x_p=0, d_p(t_2)>d^*)$.

\end{proof}

\end{appendix}
\end{document}